\documentclass[12pt,reqno]{amsart}

\numberwithin{equation}{section}

\usepackage{amsbsy}

\usepackage[colorlinks]{hyperref}
\hypersetup{
linkcolor=blue,          
citecolor=green,        
}
\usepackage[nobysame,abbrev,alphabetic]{amsrefs}

\usepackage{enumerate}
\usepackage{amssymb}
\usepackage{amsmath}
\usepackage{amscd}
\usepackage{amsthm}
\usepackage{amsfonts}
\usepackage{graphicx}
\usepackage[all]{xy}
\usepackage{verbatim}
\usepackage{hyperref}
\usepackage{gensymb}
\usepackage{mathrsfs}

\newtheorem{theorem}{Theorem}[section]

\newtheorem{lemma}[theorem]{Lemma}
\newtheorem{corollary}[theorem]{Corollary}
\newtheorem{cor}[theorem]{Corollary}
\theoremstyle{definition}\newtheorem{definition}[theorem]{Definition}

\newtheorem{conjecture}[theorem]{Conjecture}

\newtheorem{proposition}[theorem]{Proposition}

\theoremstyle{definition}

\theoremstyle{definition}
\theoremstyle{definition}\newtheorem{remark}[theorem]{Remark}
\theoremstyle{definition}

\newcommand{\al}{\alpha}

\newcommand{\ga}{\gamma}
\newcommand{\Ga}{\Gamma}
\newcommand{\del}{\delta}
\newcommand{\Del}{\Delta}

\newcommand{\Lam}{\Lambda}

\newcommand{\Om}{\Omega}
\newcommand{\vphi}{\varphi}
\newcommand{\vre}{\varepsilon}
    \newcommand{\Xn}{{\mathcal{L}_n}}
\newcommand\crly[1]{\mathscr{#1}}
\newcommand{\sm}{\smallsetminus}

\newcommand{\df}{{\, \stackrel{\mathrm{def}}{=}\, }}

\newcommand\Name[1]{\label{#1}{\ifdraft{\sn
      [#1]}\else\ignorespaces\fi}}

\newcommand\eq[2]{{\ifdraft{\ \tt
      [#1]}\else\ignorespaces\fi}\begin{equation}\label{#1}{#2}\end{equation}} 
\newcommand {\equ}[1]{\eqref{#1}}

\newcommand{\BB}{{\mathcal{B}}}

\newcommand\Vol{\mathrm{Vol}}

\newcommand\covrad{\mathrm{covrad}}
\newcommand\conv{\mathrm{conv}}
\newcommand\supp{\mathrm{supp}}

\newcommand{\cA}{\mathcal{A}}
\newcommand{\cB}{\mathcal{B}}

\newcommand{\cD}{\mathcal{D}}
\newcommand{\cE}{\mathcal{E}}

\newcommand{\cI}{\mathcal{I}}

\newcommand{\cL}{\mathcal{L}}
\newcommand{\cM}{\mathcal{M}}

\newcommand{\cS}{\mathcal{S}}

\newcommand{\cU}{\mathcal{U}}
\newcommand{\cV}{\mathcal{V}}
\newcommand{\cW}{\mathcal{W}}

\newcommand{\cY}{\mathcal{Y}}

\newcommand{\spa}{{\rm span}}

\newcommand{\LL}{\mathcal{L}}

\newcommand{\bR}{\mathbb{R}}
\newcommand{\bZ}{\mathbb{Z}}

\newcommand{\R}{{\mathbb{R}}}

\newcommand{\Z}{{\mathbb{Z}}}

\newcommand {\ignore}[1]  {}

\newcommand{\SL}{\operatorname{SL}}

\newcommand{\defi}{\overset{\on{def}}{=}}

\newcommand\norm[1]{\left\|#1\right\|}

\newcommand\set[1]{\left\{#1\right\}}
\newcommand\pa[1]{\left(#1\right)}

\newcommand{\E}{\mathbf{e}}

\newcommand\av[1]{\left|#1\right|}
\newcommand\on[1]{\operatorname{#1}}

\newcommand\tb[1]{\textbf{#1}}
\newcommand\mat[1]{\pa{\begin{matrix}#1\end{matrix}}}
\newcommand\br[1]{\left[#1\right]}
\newcommand\smallmat[1]{\pa{\begin{smallmatrix}#1\end{smallmatrix}}}

\newcommand{\lra}{\longrightarrow}

\newcommand{\onto}{\xymatrix{\ar@{>>}[r]&}}
\newcommand{\da}[4]{\xymatrix{#1 \ar@<.5ex>[r]^{#2} \ar@<-.5ex>[r]_{#3} & #4}}

\newif\ifdraft\drafttrue
\draftfalse

\font\sn = cmssi8 scaled \magstep0

\newcommand{\nerve}{{\rm Nerve}}
\newcommand{\order }{{\mathrm {ord}}}
\newcommand{\Lb}{{\mathrm {Leb}}}
\newcommand{\mesh}{{\mathrm {mesh}}}

\newcommand{\asdim}{{\rm asdim}}

\begin{document}
\title{Stable lattices and the diagonal group}
\author{Uri Shapira}
\address{Dept. of Mathematics, Technion, Haifa, Israel
{\tt ushapira@tx.technion.ac.il} 
}
\author{Barak Weiss}
\address{Dept. of Mathematics, Tel Aviv University, Tel Aviv, Israel
{\tt barakw@post.tau.ac.il}}

\maketitle
\begin{abstract}
Inspired by work of McMullen, we show that any orbit of the diagonal
group in the space of lattices accumulates on the set of stable
lattices. As consequences, we settle a conjecture of Ramharter
concerning the asymptotic behavior of the Mordell constant, and
reduce Minkowski's conjecture on products of linear forms to a
geometric question, yielding two new proofs of the conjecture in
dimensions up to 7. 

\end{abstract}
\section{Introduction}
Let $n \geq 2$ be an integer, let $G \df \SL_n(\R), \, \Gamma \df
\SL_n(\Z)$, let $A \subset G$ be the subgroup of positive diagonal
matrices and let  $\Xn \df G/\Gamma$ be the space of unimodular 
lattices in $\R^n$. The purpose of this paper is to present a
dynamical result regarding the action of $A$ on $\Xn$, and to present
some consequences in the geometry of numbers. 

A lattice $x \in \Xn$ is called {\em stable} if for any
subgroup $\Lambda \subset x$, the covolume of $\Lambda$ in
$\spa(\Lambda)$ is at least 1. In particular the length of the
shortest nonzero vector in $x$ is at least 1. Stable lattices have
also been called `semistable', they were introduced in a broad
algebro-geometric context by Harder, Narasimhan and Stuhler
\cite{Stuhler, Harder}, and were used to develop a
reduction theory for the study of the topology of locally symmetric
spaces. See Grayson \cite{Grayson} for a clear exposition. 
\begin{theorem}\Name{thm: main} 
For any $x \in \Xn$, the orbit-closure $\overline{Ax}$ contains a
stable lattice. 
\end{theorem}

Theorem \ref{thm: main} is inspired by
a breakthrough result of McMullen \cite{McMullenMinkowski}. Recall that a lattice in $\Xn$ is
called {\em well-rounded} if its shortest nonzero vectors span $\R^n$.
In connection with his work on
Minkowski's conjecture, McMullen showed that the closure of any bounded
$A$-orbit in $\Xn$ contains a well-rounded lattice. The
set of well-rounded lattices neither contains, nor is contained in,
the set of stable lattices; while the set of well-rounded lattices has
no interior, the set of stable lattices does, and in fact it occupies
all but an exponentially small volume of $\Xn$ for large $n$. Our proof of Theorem \ref{thm:
  main} closely follows McMullen's. Note however that we do not 
 assume that $Ax$ is bounded. 

We apply Theorem \ref{thm: main} to two problems in the geometry of
numbers. 
Let $x \in \Xn$ be a unimodular lattice. By a {\em symmetric box} in
$\R^n$ we mean a set of the form $ [-a_1, a_1] \times \cdots \times
[-a_n, a_n]$, and we say that a symmetric box is 
 {\em admissible} for $x$ if it contains no nonzero points of
 $x$ in its interior.  The {\em Mordell constant} of $x$
 is defined to be  
\eq{eq: defn const}{
\kappa(x) \df \frac{1}{2^n} \sup_{\BB} \Vol(\BB), 
}
where the supremum is taken over admissible symmetric boxes $\BB$, and
where $\Vol(\BB)$ denotes the volume of $\BB$.
We also write 
\eq{eq: defn kappan}{\kappa_n \df \inf\{\kappa(x): x \in \Xn\}.
}
The infimum in this definition is in fact a minimum, and, as with many
problems in the geometry of numbers it is of interest to compute the
constants $\kappa_n$ and identify the lattices realizing the
minimum. However this appears to be a very difficult problem, which so
far has only been solved for $n=2,3$, the latter in a difficult paper
of Ramharter \cite{Ramharter_dim3}. It is also of interest to provide
bounds on the asymptotics of $\kappa_n$, and in \cite{Ramharter_conjecture},
Ramharter conjectured that $\limsup_{n \to \infty} \kappa_n^{1/n\log
  n}>0$. As a simple corollary of Theorem \ref{thm: main}, we validate
Ramharter's conjecture, with an explicit bound: 
\begin{cor}\Name{cor: Ramharter conj}
For all $n \geq 2,$ 
\eq{eq: our bound}{\kappa_n \geq n^{-n/2}.
} 
In particular
$$
\kappa_n^{1/n\log
  n}  \geq n^{-1/2\log n}  \longrightarrow_{n \to \infty} 
 \frac{1}{\sqrt{e}}. 
$$
\end{cor}
We remark that Corollary \ref{cor: Ramharter conj} could also be
derived from McMullen's results and a theorem of Birch and
Swinnerton-Dyer. We refer the reader to
\cite{gruber} for more information on the possible values of
$\kappa(x), x \in \Xn$, and to the preprint \cite[\S4]{SW3_arxiv} for
slight improvements. 

Our second application concerns Minkowski's conjecture\footnote{It is not clear to us whether
  Minkowski actually made this conjecture.}, 
which posits that for any unimodular lattice $x$, one has 
\eq{eq: Minkowski conj}{
\sup_{u \in \R^n} \, \inf_{v \in x} |N(u-v)| \leq \frac{1}{2^n},
}
where $N(u_1, \ldots, u_d) \df \prod_j u_j.$
Minkowski solved the question for $n=2$ and several
authors resolved the cases $n \leq 5$. 
In \cite{McMullenMinkowski}, McMullen settled the case
$n=6$. In fact, using his theorem on the $A$-action on $\Xn$, McMullen
showed that in arbitrary dimension $n$, 
Minkowski's conjecture is implied by the statement that any well-rounded
lattice $x \subset \R^d$ with $d \leq n$ satisfies 
\eq{eq: covrad}{\covrad(x) \leq \frac{\sqrt{d}}{2},}
where $\covrad(x) \df \max_{u \in \R^d} \min_{v \in x} \|u-v\|$
and $\| \cdot \|$ is the Euclidean norm on $\R^d$. At the time of
writing \cite{McMullenMinkowski}, \equ{eq:
  covrad} was known to hold for well-rounded lattices in dimension at most $6$, and in recent work of
Hans-Gill, Raka, Sehmi and Leetika \cite{hans-gill1, hans-gill2, leetika}, \equ{eq: covrad} has
been proved for well-rounded lattices in dimensions $n=7,8,9$, thus settling Minkowski's question in
those cases. 

Our work gives two new  approaches to Minkowski's conjecture, and each
of 
these approaches yields a new proof of the conjecture in dimensions
$n\leq 7$. A direct
application of Theorem \ref{thm: main} (see Corollary 
\ref{cor: for Minkowski 1}) shows that it follows in 
dimension $n$, from the assertion that for any stable $x \in \Xn$, \equ{eq: covrad} holds. Note
that we do not require \equ{eq: covrad} in dimensions less than
$n$. Using the strategy of Woods and Hans-Gill et al, in Theorem
\ref{thm: use of KZ diagonal} we define a compact subset  $\mathrm{KZS} \subset \R^n$ and a
collection 
of $2^{n-1}$ subsets $\{ \cW(\cI)\}$ of $ \R^n$. We show that the
assertion $\mathrm{KZS} \subset \bigcup_{\cI} \cW(\cI)$ implies
Minkowski's conjecture in dimension $n$. This provides a computational
approach to Minkowski's conjecture.

Secondly, an induction using the naturality of stable lattices,
leads to the following sufficient condition: 

\begin{cor}\Name{cor: Minkowski}
Suppose that for some dimension $n$, for all $d\leq n$, any stable
lattice $x \in \LL_{d}$ which is a local maximum of the function covrad, satisfies
\equ{eq: covrad}. Then \equ{eq: Minkowski conj} holds for any $x \in \Xn$. 
\end{cor}
The local maxima of the function covrad have been
studied in depth in recent work of Dutour-Sikiri\'c, Sch\"urmann
and Vallentin \cite{mathieu}, who characterized them and showed that there
are finitely many in each dimension. Dutour-Sikiri\'c has formulated
a Conjecture as to which of these have the largest covering radius
(see Conjecture \ref{conjecture: mathieu}), and has verified his conjecture
computationally in dimensions $n \leq 7$. Our results imply that
Minkowski's conjecture is a consequence of Conjecture \ref{conjecture:
  mathieu}.

\subsection{Acknowledgements} Our work was inspired by Curt McMullen's
breakthrough paper \cite{McMullenMinkowski}
and many of our arguments are adaptations of
arguments appearing in \cite{McMullenMinkowski}. 
We are also grateful to Curt McMullen for additional insightful remarks, and in
particular for the suggestion to study the set of stable lattices in
connection with the $A$-action on $\Xn$. We also thank Mathieu
Dutour-Sikiri\'c, Rajinder Hans-Gill, G\"unter Harder, Gregory Minton and
Gerhard Ramharter for useful discussions. 

We are grateful to the referee for helping us  improve the
presentation of our results. 
A previous version of this paper, which included several other results, was
circulated under the title `On stable lattices and the diagonal
group.' At the referee's suggestion, the current version
presents our main results but omits others. For the original version
the reader is referred to \cite{SW3_arxiv}. Additional results will
appear elsewhere. 

The authors' work was
supported by ERC starter grant DLGAPS 279893, the Chaya fellowship, and ISF grants 190/08 and
357/13.

\section{Orbit closures and stable lattices}
Given a lattice $x\in \Xn$ and a subgroup $\Lam \subset x$, we denote by
$r(\Lam)$ the rank of $\Lam$ and by $\av{\Lam}$ the covolume of $\Lam$
in the linear subspace $\spa(\Lambda)$. 
Let 
\begin{align}\label{alpha}
\nonumber \cV(x)&\defi\set{\av{\Lam}^{\frac{1}{r(\Lam)}}: \Lam \subset
  x 
}, \\
\al(x)&\defi\min\cV(x).
\end{align}
Since we may take $\Lam = x$ we have 
$\alpha(x) \leq 1$ for all $x \in \Xn$, and $x$ is stable precisely if $\alpha(x)=1$. 
Observe that $\cV(x)$ 
is a countable discrete subset of the positive reals, and hence the
minimum in 
\eqref{alpha} is attained. 
Also note that the function $\al$ is a variant of the `length of the shortest
  vector'; it is continuous and the sets $\{x: \alpha(x) \geq \vre\}$
  are an exhaustion of $\Xn$ by compact sets. 

We begin by explaining the strategy for proving Theorem \ref{thm:
  main}, which is identical to the one used by McMullen. 
For a lattice $x\in X$ and $\vre>0$  we define an open cover  
$\cU^{x,\vre}=\set{U^{x,\vre}_k}_{k=1}^n$ 
of the diagonal group $A$, where if $a\in U^{x,\vre}_k$ then $\al(ax)$
is `almost attained' by a subgroup of rank $k$. In particular,  
if $a\in U^{x,\vre}_n$ then $ax$ is `almost stable'. 
The main point is to show that for any $\vre>0$, $U^{x,\vre}_n \neq
\varnothing$; for then, taking $\vre_j \to 0$ and $a_j \in A$ such
that $a_j\in U_n^{x,\vre_j}$, we find (passing to a subsequence) that
$a_jx$ converges to a stable lattice. 

In order to establish that
$U_n^{x,\vre}\ne\varnothing$, we apply a topological result of McMullen
(Theorem~\ref{topological input}) regarding open covers  
which is reminiscent of the classical result of Lebesgue
that asserts that in an open cover of Euclidean $n$-space by bounded balls
there must be a point which is covered $n+1$ times. We will work to
show that the 
cover $\cU^{x,\vre}$ satisfies the assumptions of
Theorem~\ref{topological input}. We will be able to verify these assumptions
when the orbit $Ax$ is bounded. In~\S\ref{sec: reduction to compact orbits} we reduce the proof of
Theorem~\ref{thm: main} to this case.

\subsection{Reduction to bounded orbits}\Name{sec: reduction to compact orbits}
Using a result of Birch and Swinnerton-Dyer, we
will now show that it suffices to prove  
Theorem~\ref{thm: main}
under the assumption that the orbit $Ax\subset \Xn$ is bounded; that is,
that $\overline{Ax}$ is compact. 
In this subsection we will denote $A,G$ by $A_n, G_n$ as various dimensions will appear. 

For a matrix $g\in G_n$ we denote by $\br{g}\in \Xn$ the corresponding lattice. If 
\eq{block form}{
g=\mat{g_1&*&\dots&*\\ 0& g_2&\dots&\vdots \\ \vdots& &\ddots&* \\ 0&\dots&0&g_k}
}
where $g_i\in G_{n_i}$ for each $i$, then we say that $g$ is in
\textit{upper triangular block form} 
and refer to the $g_i$'s as the \textit{diagonal blocks}. Note 
that in this definition, we insist that
each $g_i$ is of determinant one. 

\begin{lemma}\Name{lem: block stable is stable}
Let $x=\br{g}\in \Xn$  where $g$ is in upper triangular block form as
in~\eqref{block form} and for each $1\le i\le k$, $\br{g_i}$ is  
a stable lattice in $\mathcal{L}_{n_i}$. Then $x$ is stable.
\end{lemma}
\begin{proof}
By induction, in proving the Lemma we may assume that $k=2$. Let us
denote the standard basis of $\R^n$ by $\E_1, \ldots, \E_n$, let
us write $n=n_1+n_2$, 
$V_1 \df\on{span}\set{\E_1, \ldots, \E_{n_1}}$,
$V_2 \df \on{span}\set{\E_{n_1+1} \ldots, \E_n}$, and let $\pi: \R^n
\to V_2$ be the natural projection. By construction we have $x \cap
V_1 = [g_1], \pi(x) = [g_2]$. 
 
Let $\Lam \subset x$ be a subgroup, write 
$\Lam_1 \df \Lam\cap V_1$ and choose a direct complement 
$\Lam_2 \subset \Lam$, that is 
$$\Lam=\Lam_1+\Lam_2, \ \ \Lam_1 \cap \Lam_2 = \{0\}.$$ 
We claim that 
\eq{eq: claim 1}{\av{\Lam}=\av{\Lam_1}\cdot\av{\pi(\Lam_2)}.}
To see this we recall that one may compute
$|\Lam|$ via the Gram-Schmidt process. Namely, one begins with a set of generators $v_j$
of $\Lambda$ and successively defines $u_1=v_1$ and $u_j$ is the
orthogonal projection of $v_j$ on $\spa (v_1, \ldots,
v_{j-1})^\perp$. In these terms, $|\Lam| = \prod_j \|u_j\|$. Since $\pi$ is an orthogonal
projection and $\Lam \cap V_1$ is in $\ker \pi$, \equ{eq: claim 1} is clear from the
above description. 

The discrete subgroup
$\Lam_1$, when viewed as a subgroup of $\br{g_1}\in \mathcal{L}_{n_1}$ satisfies
$\av{\Lam_1}\ge 1$ because $\br{g_1}$ is assumed to be
stable. Similarly $\pi(\Lam_2) \subset [g_2] \in \mathcal{L}_{n_2}$
satisfies $\av{\pi(\Lam_2)}\ge 1$, hence 
$\av{\Lam}\ge 1$. 
\end{proof}
\begin{lemma}\Name{lem: compt red}
Let $x\in \Xn$ and assume that $\overline{Ax}$ contains a lattice
$\br{g}$ with $g$ of upper triangular block form 
as in~\eqref{block form}. For each $1\le i\le k$,  
suppose $\br{h_i}\in\overline{A_{n_i}\br{g_i}}\subset \mathcal{L}_{n_i}$. Then there
exists a lattice $\br{h}\in\overline{Ax}$ such that $h$ has the
form~\eqref{block form} with $h_i$ as its diagonal blocks. 
\end{lemma}
\begin{proof}
Let $\Omega$ be the set of all lattices $[g]$ of a fixed triangular
form as in ~\eqref{block form}. Then $\Omega$ is a closed subset of
$\Xn$ and there is a projection 
$$\tau: \Omega \to \mathcal{L}_{n_1} \times
\cdots \times \mathcal{L}_{n_k}, \ \ \tau(\br{g}) =
(\br{g_1},\dots,\br{g_k}).$$ 
The map $\tau$ has a compact fiber and is equivariant with respect to the action
of $\widetilde{ A} \df A_{n_1} \times \cdots \times A_{n_k}$. 
By assumption, there is a sequence  $\tilde{a}_j = \left(a^{(j)}_1,
\ldots, a^{(j)}_k\right), \ a^{(j)}_i \in A_{n_i}$ in
$\widetilde{ A}$ such that $a^{(j)}_i [g_i] \to [h_i]$, then after passing to
a subsequence, $\tilde{a}_j [g] \to [h]$ where $h$ has the required
properties. Since $\overline{Ax} \supset \overline {\widetilde{A}[g]}$, the claim
follows. 
\end{proof}
\begin{lemma}\label{BSD}
Let $x\in \Xn$. Then there is $[g] \in \overline{Ax}$ such that, up to
a possible permutation of the coordinates,
$g$ is of upper triangular block form as in~\eqref{block form} and 
each $A_{n_i}\br{g_i}\subset \mathcal{L}_{n_i}$ is bounded.
\end{lemma}
\begin{proof}
If the orbit $Ax$ is bounded there is nothing to prove. According to
Birch and Swinnerton-Dyer \cite{BirchSD}, if $Ax$ is
unbounded 
then $\overline{Ax}$ contains a lattice with a
representative as in~\eqref{block form} (up to a possible permutation
of the coordinates) with $k=2$. Now the claim follows using 
induction and appealing to
Lemma~\ref{lem: compt red}. 
\end{proof}
\begin{proposition}\label{copt red prop}
It is enough to establish Theorem~\ref{thm: main} for
lattices having a bounded $A$-orbit.
\end{proposition}
\begin{proof}
Let $x\in \Xn$ be arbitrary. By Lemma~\ref{BSD}, $\overline{A x}$
contains a lattice $\br{g}$ with $g$ of upper triangular block form
(up to a possible permutation of the coordinates)
with diagonal blocks representing lattices with bounded orbits under
the corresponding diagonal groups. Assuming Theorem~\ref{thm: main}
for lattices having bounded orbits, and applying Lemma~\ref{lem: compt
  red} we may take $g$ whose diagonal blocks represent  
stable lattices. By Lemma~\ref{lem: block stable is stable}, $\br{g}$ is
stable as well.
\end{proof}

\subsection{Some technical preparations}
We now discuss the  subgroups of a lattice $x\in \Xn$ which
almost attain the minimum $\al(x)$ in~\eqref{alpha}.  
 \begin{definition}\label{bn}
Given a lattice $x\in \Xn$ and $\del>0$, let 
\begin{align*}
\on{Min}_{\del}(x)&\defi\set{\Lam \subset x:\av{\Lam}^{\frac{1}{r(\Lam)}}<(1+\del)\al(x)},\\
\tb{V}_{\del}(x)&\defi\on{span}\left( \bigcup \left\{ \Lambda: \Lambda \in
  \on{Min}_{\del}(x) \right \} \right),\\
\dim_\del(x)&\defi\dim\tb{V}_{\del}(x).
\end{align*}
\end{definition}
We will need the following technical statement. 
\begin{lemma}\label{for the inradius}
For any $\rho>0$ 
there exists a
neighborhood of the identity $W\subset 
G$ with the  
following property. Suppose  $ 2\rho \leq \delta_0 \leq
d+1$ and suppose  $x\in \Xn$ is 
such that 
$\dim_{\delta_0-\rho}(x)=\dim_{\delta_0+\rho}(x)$. 
Then for any $g\in W$ and any 
$\del\in \left(\del_0-\frac{\rho}{2},\del_0+\frac{\rho}{2} \right)$ we have 
\begin{equation}\label{eq 1806}
\tb{V}_{\del}(gx)=g\tb{V}_{\del_0}(x).
\end{equation}
In particular, there is $1 \leq k \leq n$ such that  
for any $g\in W$ and any $\del\in
\left(\del_0-\frac{\rho}{2},\del_0+\frac{\rho}{2} \right)$, 
$\dim_\del(gx)=k$. 
\end{lemma}
\begin{proof}
Let $c>1$ be chosen close enough to 1 so that for $2\rho \leq \delta_0
\leq d+1$ we have
\eq{eq: defn c}{c^2\left(1+\del_0+\frac{\rho}{2} \right) < 1+\del_0 +\rho \ \ \text{and
} \ \frac{1+\del_0-\frac{\rho}{2}}{c^2} > 
1+\del_0-\rho.}
Let $W$ be a small enough neighborhood of the identity
in $G$, so that for any discrete subgroup $\Lam \subset \bR^n$ we have 
\begin{equation}\label{22.2.2}
g\in W \ \ \implies \ \ c^{-1}\av{\Lam}^\frac{1}{r(\Lam)}\le
\av{g\Lam}^\frac{1}{r(g\Lam)}\le c\av{\Lam}^\frac{1}{r(\Lam)}. 
\end{equation}
Such a neighborhood exists since  the linear action of $G$
on $\bigoplus_{k=1}^n\bigwedge^k_1 \R^n$ is continuous, and since 
we can write $|\Lam| = \|v_1 \wedge \cdots \wedge v_r\|$
where $v_1, \ldots, v_r$ is a generating set for $\Lam$.
%
%
It follows from~\eqref{22.2.2} that for any $x\in \Xn$ and $g\in W$ we have 
\eq{22.2.3}{
c^{-1}\al(x)\le \al(gx)\le c\al(x).
}
 Let $\del\in \left(\del_0-\frac{\rho}{2}, \del_0+\frac{\rho}{2}
 \right)$ and $g\in W$. We will show below that 
 \begin{equation}\label{22.2.1'}
g\on{Min}_{\del_0-\rho}(x)\subset
\on{Min}_{\del}(gx)\subset
g\on{Min}_{\del_0+\rho}(x). 
\end{equation}
Note first that 
\equ{22.2.1'} implies the assertion of the Lemma; indeed, since
$\tb{V}_{\del_1}(x) \subset \tb{V}_{\del_2}(x)$ for $\delta_1 < \delta_2$, and since we assumed
that 
$\dim_{\delta_0-\rho}(x)=\dim_{\delta_0+\rho}(x)$, 
we see that 
$\tb{V}_{\del_0}(x)=\tb{V}_{\del}(x)$ for $\delta_0-\rho \leq \delta
\leq \delta_0+\rho$. So by \equ{eq: defn c}, the
subspaces spanned by the two sides of \eqref{22.2.1'}
are equal to $g\tb{V}_{\del_0}(x)$ and \eqref{eq 1806}
follows. 

It remains to prove~\eqref{22.2.1'}. Let
$\Lam\in\on{Min}_{\del_0-\rho}(x)$. Then we find 
\[
\begin{split}
\av{g\Lam}^{\frac{1}{r(g\Lam)}} & \stackrel{\eqref{22.2.2}}{\leq}
c\av{\Lam}^{\frac{1}{r(\Lam)}} \leq c(1+\delta_0 -\rho) \alpha(x) \\ &
\stackrel{\equ{eq: defn c}}{\le}
c^{-1}\left(1+\del_0-\frac{\rho}{2} \right)\al(x) \stackrel{\equ{22.2.3}}{<}(1+\del)\al(gx).
\end{split}\]
By definition this means that $g\Lam\in\on{Min}_{\del}(gx)$ which
establishes the first  inclusion in \eqref{22.2.1'}. The second
inclusion is similar and is left to the reader.  
\ignore{
For the second inclusion, let $\Lam \subset x$ such that
$g\Lam\in\on{Min}_{(\del)}(gx)$. Using the definition and~\eqref{22.2.2},\eqref{22.2.3} we conclude that 
\[\av{\Lam}^{\frac{1}{r(\Lam)}} \leq c\av{g\Lam}^{\frac{1}{r(g\Lam)}}\le
c(1+\del)\al(gx)< c^2 (1+\del_0+\frac{\rho}{2})\al(x).\]
I.e.\ $\Lam\in\on{Min}_{c^2(1+\del_0+\frac{\rho}{2})}(x)$ which
establishes the right inclusion in~\eqref{22.2.1'}. }
\end{proof}

\subsection{The cover of $A$} 
Let $x\in \Xn$ and let $\vre>0$ be given. Define
$\cU^{x,\vre}=\left\{U^{x,\vre}_i \right\}_{i=1}^n$ where 
\begin{equation}\label{the cover}
U^{x,\vre}_k\defi\set{a\in A: \on{dim}_\del(ax)=k\textrm{ for $\del$
    in a neighborhood of }k\vre}. 
\end{equation} 
\begin{theorem}\Name{order of cover}
Let $x\in \Xn$ be such that $Ax$ is bounded. Then for any $\vre \in (0,1)$, 
$U^{x,\vre}_n\neq \varnothing.$ 
\end{theorem}
In this subsection we will reduce the proof of Theorem \ref{thm: main}
to Theorem \ref{order of cover}. This will be done via the following
statement, 
which could be interpreted as saying that a 
lattice satisfying $ \dim_{\delta}(x)=n
$ is `almost stable'. 

\begin{lemma}\label{not growing lemma}
For each $n$, there exists a positive function $\psi(\del)$ with
$\psi(\del) \to_{\del\to 0}0$, such that for any $x\in \Xn$,
\eq{eq: lemma first part}{\set{\Lam_i}_{i=1}^\ell\subset\on{Min}_{\del}(x) \ \implies \ 
\Lam_1+\dots+\Lam_\ell\in\on{Min}_{\psi(\del)}(x).
}
In particular, if $\dim_\del(x)=n$ then $\al(x)\ge (1+\psi(\del))^{-1}$.  
\end{lemma}
\begin{proof}
Let $\Lam,\Lam'$ be two discrete subgroups of $\bR^d$.
The following inequality is straightforward to 
prove via the Gram-Schmidt procedure for computing $|\Lam|$: 
\begin{equation}\label{volume formula}
\av{\Lam+\Lam'}\le\frac{\av{\Lam}\cdot\av{\Lam'}}{\av{\Lam\cap\Lam'}}.
\end{equation}
Here we adopt the convention that  $\av{\Lam\cap\Lam'}=1$ when
$\Lam\cap\Lam'=\set{0}$. 
By induction on $\ell \leq n$, we now prove 
the existence of a
function  $\psi_\ell(\del)\overset{\del\to0}{\lra}0$ such that for any
$x \in \Xn$ and any $\set{\Lam_i}_{i=1}^\ell\subset
\on{Min}_{\del}(x)$, we have 
$\Lam_1+\dots+\Lam_\ell\in\on{Min}_{\psi_\ell(\del)}(x)$. For
$\ell=1$ one can trivially pick $\psi_1(\del)=\del$.  
Assuming the existence of $\psi_{\ell-1}$,  set 
$$\psi_\ell(\del)\defi \max
\pa{(1+\del)^{r(\Lam)}(1+\psi_{\ell-1}(\del))^{r(\Lam')}}^{\frac{1}{r(\Lam+\Lam')}}-1,$$
where  
the maximum is taken over all possible values of $r(\Lam), r(\Lam'),
r(\Lam+\Lam')$. Clearly $\psi_\ell(\del)\lra_{\del\to 0}0$, and given
$x \in \Xn$ and $\Lam_1, \ldots, \Lam_{\ell} \in 
\on{Min}_{\del}(x),$ set $\Lam=\Lam_1$,
$\Lam'=\Lam_2+\dots+\Lam_\ell$, $\al=\al(x)$ and note  
that $r(\Lam+\Lam')=r(\Lam)+r(\Lam')-r(\Lam\cap\Lam')$. We deduce
from~\eqref{volume formula} and the definitions that 
\begin{align*}
\nonumber \av{\Lam+\Lam'}&\le
\frac{\av{\Lam}\cdot\av{\Lam'}}{\av{\Lam\cap\Lam'}}
\le
\frac{\pa{(1+\del)\al}^{r(\Lam)}\pa{(1+\psi_{\ell-1}(\del))\al}^{r(\Lam')}}{\al^{r(\Lam\cap\Lam')}}\\  
&=(1+\del)^{r(\Lam)}(1+\psi_{\ell-1}(\del))^{r(\Lam')}\al^{r(\Lam+\Lam')},
\end{align*}
and so 
$\Lam+\Lam'\in\on{Min}_{\psi_\ell(\del)}(x)$ as desired. This
completes the inductive step.

We take
$\psi(\del)\defi\max_{\ell=1}^n\psi_\ell(\del).$ If $\ell \leq n$ then
\equ{eq: lemma first part} holds by construction. If $\ell >n$  one can find a subsequence 
$1\le i_1<i_2\dots<i_d\le n$  such that
$r(\sum_{i=1}^\ell\Lam_i)=r(\sum_{j=1}^d\Lam_{i_j})$ and in
particular,  
$\sum_{j=1}^d\Lam_{i_j}$ is of finite index in $\sum_{i=1}^\ell\Lam_i$. From the first part of the 
argument we see that $\sum_{j=1}^d\Lam_{i_j}\in \on{Min}_{\psi(\del)}(x)$ and as the covolume of 
$\sum_{i=1}^\ell\Lam_i$ is not larger than that of $\sum_{j=1}^d\Lam_{i_j}$ we deduce that 
$\sum_{i=1}^\ell\Lam_i \in\on{Min}_{\psi_\ell(\del)}(x)$ as well.

To verify the last assertion, note that
when $\dim_\del(x)=n$, \equ{eq: lemma first part} implies the
existence of a finite index subgroup $x'$ 
of $x$ belonging to $\on{Min}_{\psi(\del)}(x)$. In particular, 
$1\leq \av{x'}^{\frac{1}{n}}\le(1+\psi(\del))\al(x)$ as desired.
\end{proof}

\ignore{
The following statement essentially 
says that a lattice $x$ satisfying $\dim_\del(x)=n$ with $\del$ small
can be considered to be `almost stable'.
\begin{corollary}\label{hitting stable cor}
If $x_j\in \Xn$ satisfies $\dim_{\del_j}(x_j)=n$ for some sequence
$\del_j \to 0$, then any accumulation point of 
$\set{x_j}$ is a stable lattice.
\end{corollary}
\begin{proof}
By Lemma~\ref{not growing lemma} we have  
$$1\ge
\limsup\al(x_j)\ge \liminf \al(x_j)\ge \lim (1+\psi(\del_j))^{-1}=1,$$ 
which shows that $\lim\al(x_j)=1$.  
The function $\al$ is continuous on $\Xn$ and therefore if $x$  is an
accumulation point of $\set{x_j}$ then $\al(x)=1$, i.e. $x$ is stable.
\end{proof}
}
\begin{proof}[Proof of Theorem~\ref{thm: main} assuming
  Theorem~\ref{order of cover}] 
By Proposition~\ref{copt red prop} we may assume that $Ax$ is
bounded. Let $\vre_j \in (0,1)$ so that $\vre_j \to_j 0$. By
Theorem~\ref{order of cover} we know that
$U^{x,\vre_j}_n 
\neq \varnothing$. This means there is a sequence $a_j\in A$ such that 
$\dim_{\del_j}(a_jx)=n$
where $\del_j=n\vre_j\to 0$. 
The sequence $\set{a_jx}$ is
  bounded, and hence has limit points, so passing to a subsequence we
  let $x' \df \lim a_jx.$ 
By Lemma~\ref{not growing lemma} we have  
$$1\ge
\limsup_j\al(a_jx)\ge \liminf_j \al(a_j x)\ge \lim_j (1+\psi(\del_j))^{-1}=1,$$ 
which shows that $\lim_j\al(a_j x)=1$.  
The function $\al$ is continuous on $\Xn$ and therefore $\al(x')=1$,
i.e. $x' \in \overline{Ax}$ is stable.
\end{proof}
\section{Covers of Euclidean space}\Name{establishing
  topological input}  
In this section we will prove Theorem \ref{order of cover}, thus
completing the proof of Theorem \ref{thm: main}. Our main
tool will be McMullen's 
Theorem~\ref{topological input}. Before stating it we introduce some
terminology. We fix an invariant metric on $A$, and let $R>0$ and $k \in \{0, \ldots, n-1\}$. 
\begin{definition}\Name{def: almost affine}
We say that a subset $U\subset A$ is $(R,k)$-\textit{almost affine} if it is
contained in an $R$-neighborhood of a coset of a connected $k$-dimensional 
subgroup of $A$. 
\end{definition}
\begin{definition}\Name{def: inradius}
An open cover $\cU$ of $A$ is said to have \textit{inradius} $r>0$ if
for any $a\in A$ there exists $U\in\cU$ such that $B_r(a)\subset
U$, where $B_r(a)$ denotes the ball in $A$ of radius $r$ around $a$. 
\end{definition}
\begin{theorem}[Theorem 5.1 of~\cite{McMullenMinkowski}]\Name{topological input}
Let $\cU$ be an open cover of $A$ with inradius $r>0$ and let
$R>0$. Suppose that for any $1\le k\le n-1$, 
every connected component $V$ of the intersection of  
$k$ distinct elements of $\cU$ is
$(R,(n-1-k))$-almost affine. Then there is 
a point in $A$ which belongs to  
at least $n$ distinct elements of $\cU$. In particular, there are at
least $n$ distinct non-empty sets in $\cU$. 
\end{theorem}

\ignore{
\begin{proposition}\label{assumptions hold}
If $x\in \Xn$ has a bounded $A$-orbit and $\vre>0$ then the collection
$\cU^{x,\vre}$ is an open cover of $A$ with positive inradius  
such that at least one of the following two possibilities hold:
\begin{enumerate}
\item $U^{x,\vre}_n\ne\varnothing$.
\item The hypothesis of Theorem~\ref{topological input} are satisfied.
\end{enumerate}
\end{proposition}
\begin{proof}[Proof of Theorem~\ref{order of cover} assuming
  Proposition~\ref{assumptions hold}] 
By the Proposition the possibility that $U^{x,\vre}_d=\varnothing$ is
ruled out as if this is the case then we may apply  
Theorem~\ref{topological input} and deduce that $\cU^{x,\vre}$ must
contain at least $n$ non-empty sets and in particular,  
$U_n^{x,\vre}\ne\varnothing$ which contradicts our assumption.
\end{proof} 
}

\subsection{Verifying the hypotheses of Theorem \ref{topological
    input} }
Below we fix a compact set $K\subset \Xn$ and a lattice $x$ for which
$Ax\subset K$. Furthermore, we fix $\vre>0$ and denote  
the collection $\cU^{x, \vre}$ defined in~\eqref{the cover} by
$\cU=\set{U_i}_{i=1}^n$.
\begin{lemma}\Name{lem: positive inradius}
The collection $\cU$ forms an open cover of $A$ with positive inradius.
\end{lemma}
\begin{proof}
The fact that the sets $U_i\subset A$ are open follows readily from
the requirement in~\eqref{the cover} that $\on{dim}_\del$ is constant
for $\del$ in a neighborhood of $k\vre$.  
Given $a\in A$, let $1\le k_0\le n$ be the minimal number $k$ for which
$\dim_{(k+\frac{1}{2})\vre}(ax)\le k$  
(this inequality holds trivially for $k=n$). 
From the minimality of $k_0$ we conclude that 
$\dim_\del(ax)=k_0$ for
any  
$\del\in \left[\pa{k_0-\frac{1}{2}}\vre,\pa{k_0+\frac{1}{2}}\vre \right]$. 
This
shows that $a\in U_{k_0}$ so 
$\cU$ is indeed a cover of $A$. 

We now show that the cover has positive inradius. 
Let $W \subset G$ be the open neighborhood of the identity obtained
from 
Lemma~\ref{for the inradius} for  $\rho \df \frac{\vre}{2}$.
Taking $\delta_0 \df k_0 \vre$ we find that 
for any 
$g\in W$,   
$\del\in \pa{\pa{k_0-\frac{1}{4}}\vre,\pa{k_0+\frac{1}{4}}\vre}$ we
have that $\dim_\del(gax)=k_0$. This shows that  
$(W\cap A)a\subset U_{k_0}$. Since $W\cap A$ is an open neighborhood
of the identity in $A$ and the metric on $A$ is invariant under  
translation by elements of $A$, there exists $r>0$ 
(independent of $k_0$ and $a$) so that $B_r(a)\subset U_{k_0}$. In
other words, the inradius of $\cU$ is positive as desired. 
\end{proof}

The following will be used for verifying the second hypothesis of
Theorem~\ref{topological input}.
\begin{lemma}\Name{flat things}
There exists $R>0$ such that any connected component of $U_k$ is $(R,k-1)$-almost affine.
\end{lemma}
\begin{definition}\label{cv}
For a discrete subgroup $\Lam \subset \bR^d$ of rank $k$, 
let $$c(\Lam)\defi\inf\set{\av{a\Lam}^{1/k}:a\in A},$$ and say that
$\Lam$ is {\it incompressible} if $c(\Lam)>0$. 
\end{definition}
Lemma~\ref{flat things} follows from:
\begin{theorem}[{\cite[Theorem 6.1]{McMullenMinkowski}}]\Name{finite
    distance from a group} 
For any positive $c,C$ there exists $R>0$ such that if 
$\Lam \subset \bR^n$ is an incompressible discrete subgroup of rank
$k$ with $c(\Lam)\ge c$ then  
$\set{a\in A: \av{a\Lam}^{1/k}\le C}$ is $(R,j)$-almost affine for some $j\le
\gcd(k,n)-1$. 
\end{theorem}
\begin{proof}[Proof of Lemma~\ref{flat things}]
We first claim that there exists $c>0$ such that  
for any discrete subgroup $\Lam \subset x$ we have that $c(\Lam)\ge
c$. To see this, recall that $Ax$ is contained in a compact subset
$K$, and hence by Mahler's compactness criterion,  there is a positive
lower bound on 
the length of any non-zero vector
belonging to a lattice in $K$. 
On the other
hand, Minkowski's convex body theorem shows that the shortest nonzero
vector in a discrete subgroup $\Lambda \subset \R^n$ is bounded above
by a constant multiple of $|\Lam|^{1/r(\Lam)}$. This implies the
claim.

In light of Theorem \ref{finite
    distance from a group}, it suffices to show that there is $C>0$
such that  if $V\subset U_k$ is a
connected component, then there exists 
$\Lam \subset x$ such that  $V\subset\set{a\in A: \av{a\Lam}^{1/k}\le C}$.
For any $1\le k\le n$, write $\tb{gr}_k$ for the Grassmannian of
$k$-dimensional subspaces of $\bR^n$.  
Define
$$
\cM:U_k\to \tb{gr}_k, \ \ 
\cM(a)\defi a^{-1}\tb{V}_{k\vre}(ax).$$ 
Observe that $\cM$ is locally constant on $U_k$. Indeed, by
definition of $U_k$, 
for $a_0\in U_k$ there exists 
$0<\rho< \frac{\vre}{2}$ such that 
$\dim_\del(a_0x)=k$ for any $\del\in(k\vre-\rho,k\vre+\rho)$. Applying
Lemma~\ref{for the inradius}  
for the lattice $a_0x$ with $\rho$ and $\del_0=k\vre$
we see 
that for any $a$ in a neighborhood
of the identity in $A$, 
\begin{align*}
\cM(aa_0)&=a_0^{-1}a^{-1}\tb{V}_{k\vre}(aa_0x)
=a_0^{-1}\tb{V}_{k\vre}(a_0x)=\cM(a_0).
\end{align*}

Now let $\Lam \df x\cap \cM(a)$ where $a\in V$; $\Lam$ is well-defined
since $\cM$ is locally
constant. 
Then 
for $a \in V$, 
\begin{align*}
a\Lam&=a(x\cap\cM(a)) =a(x\cap a^{-1}\tb{V}_{k\vre}(ax))=ax\cap\tb{V}_{k\vre}(ax).
\end{align*}
By Lemma~\ref{not growing lemma} we have that 
\begin{align*}
\av{a\Lam}^{1/k}=\av{
  ax\cap \tb{V}_{k\vre}(ax)}^{1/k}<(1+\psi(k\vre))\al(ax). 
\end{align*}
Since $\alpha(ax) \leq 1$ we may take $C \df 1+\psi(k\vre)$ to
complete the proof. 
%
\end{proof}

\begin{proof}[Proof of Theorem \ref{order of cover}]
Assume by contradiction that $Ax$ is bounded but $U_n^{x, \vre} =
\varnothing$ for some $\vre \in (0,1)$. Then by Lemma \ref{lem: positive inradius}, 
$$\cU \df \left \{U_1,
\ldots, U_{n-1} \right \}, \text{ where } U_j \df U_j^{x, \vre}, $$  
is a cover of $A$ of positive inradius. Moreover, if $V$ is a
connected component of $U_{j_1} \cap \cdots \cap U_{j_k}$ with $j_1 < \cdots < j_k \leq n-1$, 
then $V_k \subset U_{j_1}$ and $j_1 \leq n-k$. So in
light of Lemma \ref{flat things}, 
the hypotheses of Theorem~\ref{topological input} are satisfied.
We 
deduce that $\cU = \left\{U_1, \ldots, U_{n-1} \right \}$
contains at least $n$ elements, which is impossible. 
\end{proof}

\section{Bounds on Mordell's constant}\Name{sec: rankin bounds} 
In analogy with~\eqref{alpha} we define for
any $x\in \Xn$ and $1\le k\le n$,  
\begin{align}\Name{eq: k quantities}
\cV_k(x)&\defi\set{\av{\Lam}^{1/r(\Lam)}:\Lam \subset x, r(\Lam)=k},\\ 
\al_k(x)&\defi\min\cV_k(x). 
\end{align}

The following is clearly a consequence of Theorem \ref{thm: main}:
\begin{cor}\Name{cor: Euclidean}
For any $x \in \Xn$, any $\vre>0$  and any $k \in \{1, \ldots, n\}$ there is $a \in
A$ such that $\alpha_k(ax) \geq 1-\vre$. 
\end{cor}
As the lattice $x = \Z^n$ shows, the constant 1 appearing in
this corollary cannot be improved for any $k$. Note also that the case
$k=1$ of 
Corollary \ref{cor: Euclidean}, although not stated explicitly in
\cite{McMullenMinkowski}, could be derived easily from McMullen's results in
conjunction with \cite{BirchSD}. 

\begin{proof}[Proof of Corollary \ref{cor: Ramharter conj}]
 Since the $A$-action maps a symmetric box $\mathcal{B}$ to a
 symmetric box of the same volume, the function $\kappa : \Xn \to \R$
 in \equ{eq: defn const} is $A$-invariant. By the case $k=1$ of
 Corollary \ref{cor: Euclidean}, for any $\vre>0$ and any $x \in \Xn$
 there is $a \in A$ such that $ax$ does not contain nonzero vectors of
 Euclidean length at most $1-\vre$, and hence does not contain nonzero vectors
 in the cube $\left [-\left(\frac{1}{\sqrt{n}} - \vre\right),
 \left(\frac{1}{\sqrt{n}} - \vre\right) \right ]^n$. This implies that
$\kappa(x) \geq \left(\frac{1}{\sqrt{n}} \right)^n$, as claimed. 
\end{proof}

The bound \equ{eq: our bound} is not tight for any $n$. This is shown
in \cite{SW3_arxiv}, along with several slight improvements of
\equ{eq: our bound}. For example we prove
that if $n \geq 5$ is congruent to 1 mod 4, then 
$$
\kappa_n \geq \frac{1}{\sqrt{2n-1}(n-1)^{(n-1)/2}}.
$$
Similar slight improvements can be obtained for all $n$ not divisible
by $4$. 
See \cite{SW3_arxiv} for more details. 

\section{Two strategies for Minkowski's conjecture}\Name{sec: MC} 
We begin by recalling the well-known Davenport-Remak strategy
for proving Minkowski's conjecture. The function 
$N(u) = \prod_1^n u_i$ is clearly $A$-invariant, and 
it follows that the quantity 
$$\widetilde{ N}(x) \df  \sup_{u \in \R^n} \inf_{v
  \in x} |N(u-v)|$$ 
appearing in \equ{eq: Minkowski conj} is
$A$-invariant. Moreover, it is easy to show that if $x_j \to x$ in
$\Xn$ then $\widetilde{ N}(x) \geq \limsup_j \widetilde{ N}(x_j)$. Therefore, in order
to show the estimate \equ{eq: Minkowski conj} for $x' \in \Xn$, it is
enough to show it for some $x \in \overline{Ax'}$. Suppose that $x$
satisfies \equ{eq: covrad} with $d=n$; that is for every $u \in \R^n$ there is $v
\in x$ such that $\|u-v\| \leq \frac{\sqrt{n}}{2}$. 
Then applying  
the inequality of arithmetic and geometric means one finds
$$\prod_1^n \left(|u_i-v_i|^2 \right)^{\frac1n} \leq \frac{1}{n} \sum_1^n |u_i-v_i|^2 \leq \frac{1}{4}
$$
which implies $|N(u-v)| \leq \frac{1}{2^n}$. 
The upshot is that in order to prove Minkowski's conjecture, it is
enough to prove that for every $x' \in \Xn$ there is $x \in
\overline{Ax}$ satisfying \equ{eq: covrad}. So in light of Theorem
\ref{thm: main} we obtain:
\begin{corollary}\Name{cor: for Minkowski 1}
If all stable lattices in $\Xn$ satisfy \equ{eq: covrad}, then
Minkowski's conjecture is true in dimension $n$. 
\end{corollary}

In the next two subsections, we outline two strategies for
establishing that all stable lattices satisfy
\equ{eq: covrad}. 
Both strategies yield affirmative answers in dimensions $n
\leq 7$, thus providing new proofs of Minkowski's conjecture in these
dimensions. 

\subsection{Using Korkine-Zolotarev reduction}
Korkine-Zolotarev reduction is a classical method for
choosing a basis $v_1, \ldots, v_n$ of a lattice $x \in \Xn$. Namely
one takes for $v_1$ a shortest nonzero vector of $x$ and
denotes its length by $A_1$. Then, proceeding inductively, for $v_i$ one takes 
a vector whose projection onto $(\spa(v_1, \ldots, v_{i-1}))^\perp$ is
shortest (among those with nonzero projection), and denotes the length
of this
projection by $A_i$. In case there is more
than one shortest vector the process 
is not uniquely defined. Nevertheless we call $A_1, \ldots, A_n$ the
{\em diagonal KZ coefficients of $x$} (with the understanding that
these may be multiply defined for some measure zero subset of $\Xn$). Since $x$
is unimodular we always have 
\eq{eq: det one}{\prod A_i 
=1.}
Korkine and Zolotarev proved the bounds 
\eq{eq: KZ bounds}{
A_{i+1}^2 \geq \frac34 A_i^2, \ \ A_{i+2}^2 \geq \frac23 A_i^2.
}

A method introduced by Woods and developed further in
\cite{hans-gill1} leads to an upper bound on
$\covrad(x)$ in terms of the diagonal KZ coefficients. 
The method relies on the following estimate. Below
$\gamma_n \df \sup_{x \in \Xn} \alpha^2_1(x)$ (where $\alpha_1$
is defined via \equ{eq: k quantities}) is the 
{\em Hermite constant}. 

\begin{lemma}[\cite{Woods_n=4}, Lemma 1]\Name{lemma of Woods}
Suppose that $x$ is a lattice in $\R^n$ of covolume $d$, and suppose
that $2 A_1^n \geq
d \gamma_{n+1}^{(n+1)/2}$. Then 
$$
\covrad^2(x) \leq A_1^2 -\frac{A_1^{2n+2}}{d^2 \gamma_{n+1}^{n+1}}.
$$
\end{lemma} 
Woods also used the following observation:
\begin{lemma}[\cite{Woods_n=4}, Lemma 2]
\Name{first Woods lemma}
Let $x$ be a lattice in $\R^n$, let $\Lambda$ be a subgroup, and let
$\Lambda'$ denote the projection of $x$ onto $(\spa 
\Lambda)^{\perp}$. Then 
$$
\covrad^2(x) \leq \covrad^2(\Lambda) +\covrad^2(\Lambda')
$$
\end{lemma} 
As a consequence of Lemmas \ref{lemma of Woods} and \ref{first Woods
  lemma},  we obtain:
\begin{proposition}\Name{prop: combining Woods}
Suppose $A_1, \ldots, A_n$ are diagonal KZ coefficients of $x \in
\Xn$ and suppose $n_1, \dots, n_k$ are positive integers with $n = n_1 + \cdots
+ n_k$. 
Set
\eq{eq: defn mi di}{m_i \df n_1 + \cdots
+ n_i \ \text{and } d_i \df \prod_{j=m_{i-1}+1}^{m_{i}} A_j.}
If 
\eq{eq: assumption of woods}{
2A_{m_{i-1}+1} \geq d_i \gamma_{n_i+1}^{(n_i+1)/2} 
}
for each $i$, then 
\eq{eq: consequence}{
\covrad^2(x ) \leq \sum_{i=1}^k \left(A^2_{m_{i-1}+1} -
  \frac{A^{2n_i+2}_{m_{i-1}+1}}{d_i^2 \gamma_{n_i+1}^{n_i+1}} \right) 
}

\end{proposition}
\begin{proof}
Let $v_1, \ldots, v_n$ be the basis of $x$
obtained by the Korkine Zolotarev reduction process. Let 
$\Lambda_1$ be the subgroup of $x$ generated by $v_1, \ldots,
v_{n_1}$, and for $i=2, \ldots, k$ let
$\Lambda_i$ be the projection onto
$(\bigoplus_1^{i-1} \Lambda_j)^{\perp}$ of the subgroup of $x$
generated by $v_{m_{i-1}+1}, \ldots, v_{m_{i}}$. This is a lattice of
dimension $m_i$, and arguing as in the proof of \equ{eq: claim 1} we
see that it has covolume $d_i$. The assumption
\equ{eq: assumption of woods} says that we may apply Lemma \ref{lemma
  of Woods}
to each $\Lambda_i$.  We obtain 
$$\covrad^2(\Lambda_i) \leq
A^2_{m_{i-1}+1} - \frac{A^{2n_i+2}_{m_{i-1}+1}}{d_i^2
  \gamma_{n_i+1}^{n_i+1}}  
$$ for each $i$,  and we combine these estimates using Lemma
\ref{first Woods lemma} and an obvious induction. 
\end{proof}

\begin{remark}
Note that it is an open question to determine the numbers $\gamma_n$;
however, if we have a bound $\tilde{\gamma}_n \geq \gamma_n$ we may
substitute it into Proposition \ref{prop: combining Woods} in place of $\gamma_n$, as this
only makes the requirement \equ{eq: assumption of woods} stricter and
the conclusion \equ{eq: consequence}
weaker.  
\end{remark}

Our goal is to apply this method to the problem of bounding the covering
radius of stable lattices. We note:
\begin{proposition}\Name{prop: KZ of stable}
If $x$ is stable then we have the inequalities 
\eq{eq: defn KZ stable}{
A_1 \geq 1, \ \ A_1 A_2 \geq 1, \ \ \ldots \ \ A_1 \cdots A_{n-1} \geq 1.
}
\end{proposition}
\begin{proof}
In the above terms, the number $A_1 \cdots A_i$ is equal to
$|\Lambda|$ where $\Lambda$ is the subgroup of $x$ generated by $v_1,
\ldots, v_i$. 
\end{proof}
This motivates the following:
\begin{definition}
We say that an $n$-tuple of positive real numbers $A_1, \ldots, A_n$
is {\em KZ stable} if the inequalities \equ{eq: det one}, \equ{eq: KZ
  bounds}, \equ{eq: defn KZ stable} are 
satisfied. We denote the set of KZ stable $n$-tuples by
$\mathrm{KZS}$. 
\end{definition}
Note that $\mathrm{KZS}$ is a compact subset of
$\R^n$. Recall that a {\em composition of $n$} is an ordered $k$-tuple $(n_1,
\ldots, n_k)$ of positive integers, such that $n=n_1+\ldots +n_k$. As
an immediate application of Corollary \ref{cor: for 
  Minkowski 1} and
Propositions \ref{prop: combining Woods} and \ref{prop: KZ of stable}
we obtain:
\begin{theorem}\Name{thm: use of KZ diagonal}
For each composition $\cI \df (n_1, \ldots, n_k)$ of $n$, define $m_i, d_i$ by
\equ{eq: defn mi di} and let $\mathcal{W}(\cI)$ denote
the set 
$$\left\{(A_1, \ldots, A_n): \forall
  i, \, 
\equ{eq: assumption of woods} \text{ holds, and } \sum_{i=1}^k \left(A^2_{m_{i-1}+1} -
  \frac{A^{2n_i+2}_{m_{i-1}+1}}{d_i^2 \gamma_{n_i+1}^{n_i+1}} \right) \leq
  \frac{n}{4} \right \}.
$$
If 
\eq{eq: covering suffices}{
\mathrm{KZS} \subset \bigcup_{\cI} \cW(\cI)
}
then Minkowski's conjecture holds in
dimension $n$. 
\end{theorem}
Rajinder Hans-Gill has informed the authors  that using
arguments as in \cite{hans-gill1, hans-gill2}, it is possible to
verify \equ{eq: covering suffices} 
in dimensions up to 7, thus
reproving Minkowski's conjecture in these dimensions. 

\subsection{Local maxima of covrad}
The aim of this subsection is to prove Corollary \ref{cor: Minkowski}, 
which shows that in
order to establish that all stable lattices in $\R^n$ satisfy the covering
radius bound \equ{eq: covrad}, it suffices to check this on a finite
list of lattices in each dimension $d \leq n$.


The function $\covrad : \Xn \to \R$ is proper, but nevertheless has local maxima, in the
usual sense; that is, lattices $x \in \Xn$ for which there is a
neighborhood $\mathcal{U}$ of $x$ in $\Xn$ 
such that for all $x' \in \cU$ we have $\covrad(x') \leq
\covrad(x)$. Dutour-Sikiri\'c, Sch\"urmann and Vallentin
\cite{mathieu} gave a geometric characterization of lattices
which are local maxima of the function 
$\covrad$, and showed that there are finitely many in each dimension.
Corollary \ref{cor: Minkowski} asserts that Minkowski's conjecture
would follow if all local maxima of covrad satisfy the bound \equ{eq:
  covrad}.   
\begin{proof}[Proof of Corollary \ref{cor: Minkowski}]
We prove by induction on $n$ that any stable lattice satisfies the
bound \equ{eq: covrad} and apply Corollary \ref{cor: for Minkowski 1}. 
Let $\cS$ denote the set of stable lattices in $\Xn$. It is compact
so the function $\covrad$ attains a maximum on $\cS$, and it suffices
to show that this maximum is at most $\frac{\sqrt{n}}{2}$. Let $x \in \cS$
be a point at which the maximum is attained. If $x$ is an interior
point of $\cS$ then necessarily $x$ is a
local maximum for $\covrad$ and the required bound holds by
hypothesis. Otherwise, there is a sequence $x_j \to x$ such that
$x_j \in \Xn \sm \cS$; thus each $x_j$ contains a discrete subgroup
$\Lambda_j$ with $|\Lambda_j| <1$ and $r(\Lambda_j) <n$. Passing to a subsequence we may
assume that that $r(\Lambda_j)=k<n$ is the same for all $j$, and
$\Lambda_j$ converges to a discrete subgroup $\Lambda$ of $x$. Since
$x$ is stable we must have $|\Lambda|=1$. Let $\pi: \R^n \to (\spa
\Lambda)^{\perp}$ by the orthogonal projection and let 
$\Lambda' \df \pi(x)$. 

It suffices to show that both $\Lambda$
and $\Lambda'$ are stable. Indeed, if this holds then by the induction
hypothesis, both $\Lambda$
and $\Lambda'$ satisfy \equ{eq: 
  covrad} in their respective dimensions $k, n-k$, and by Lemma \ref{first
  Woods lemma}, so does $x$. To see that $\Lambda$ is stable, note
that any subgroup $\Lambda_0 \subset \Lambda$ is also a subgroup of
$x$, and since $x$ is stable, it satisfies $|\Lambda_0| \geq 1$.   To
see that $\Lambda'$ is stable, note that if $\Lambda_0 \subset
\Lambda'$ then $\widetilde{\Lambda_0} \df x \cap \pi^{-1}(\Lambda_0)$ is
a discrete subgroup of $x$ so satisfies $|\widetilde{\Lambda_0}| \geq
1$. Since $|\Lambda|=1$ and $\pi$ is orthogonal, we argue as in the
proof of \equ{eq: claim 1} to obtain
$$1 \leq |\widetilde{\Lambda_0}| = |\Lambda | \cdot |\Lambda_0| =
|\Lambda_0|,$$
so $\Lambda'$ is also stable, as required. 
\end{proof}

In \cite{mathieu}, it was shown that there is a unique local maximum
for covrad in dimension 1, none in dimensions 2--5, and a unique one in
dimension 6. Local maxima of covrad in dimension 7 are classified in
the manuscript \cite{mathieu2}; there are 2 such lattices. Thus
in total, in dimensions $n \leq 7$ there are 4 local maxima of the
function covrad. We
were informed by Mathieu Dutour-Sikiri\'c that these lattices all satisfy the covering radius bound
\equ{eq: covrad}. Thus Corollary \ref{cor: Minkowski} yields another proof of
Minkowski's conjecture, in dimensions $n \leq 7$. In
\cite{mathieu_solo} and infinite list of lattices, one in each
dimension $n \geq 6$, is defined. It was shown in  \cite[\S
7]{mathieu},  that each of these lattices (denoted there by $[L_n,
Q_n]$) is a local maximum for the
function covrad, and satisfies the bound \equ{eq:
  covrad}. Dutour-Sikiri\'c has conjectured:
\begin{conjecture}[M. Dutour-Sikiri\'c] \Name{conjecture: mathieu}
For each $n \geq 6$, the lattice $[L_n, Q_n]$ has the largest covering
radius among all local maxima in dimension $n$. 
\end{conjecture}
In light of Corollary \ref{cor: Minkowski}, the validity
of Conjecture \ref{conjecture: mathieu} would imply Minkowski's conjecture in all
dimensions.

\bibliographystyle{alpha}
\bibliography{elonbib}
\end{document}